\documentclass[oneside,english]{amsart}
\usepackage[T1]{fontenc}
\usepackage[latin9]{inputenc}
\usepackage{amsthm}
\usepackage{amstext}
\usepackage{amssymb}
\usepackage{esint}
\usepackage{amscd}
\makeatletter
\numberwithin{equation}{section}
\numberwithin{figure}{section}
\theoremstyle{plain}
\newtheorem{thm}{\protect\theoremname}[section]

\theoremstyle{plain}
\theoremstyle{definition}

\theoremstyle{plain}

\newtheorem{cor}[thm]{\protect\corollaryname}
\theoremstyle{plain}

\theoremstyle{plain}
\makeatother

\usepackage{babel}
\providecommand{\definitionname}{Definition}
\providecommand{\lemmaname}{Lemma}
\providecommand{\theoremname}{Theorem}
\providecommand{\corollaryname}{Corollary}
\providecommand{\remarkname}{Remark}
\providecommand{\propositionname}{Proposition}

\DeclareMathOperator{\loc}{loc}

\DeclareMathOperator{\ess}{ess}

\DeclareMathOperator{\cp}{cap}
\DeclareMathOperator{\ACL}{ACL}

\usepackage{color}

\begin{document}

\title[Refined geometric characterizations]{Refined geometric characterizations of weak $p$-quasiconformal mappings}

\author{Ruslan Salimov and Alexander Ukhlov}
\begin{abstract}
In this paper we consider refined geometric characterizations of weak $p$-quasiconformal mappings $\varphi:\Omega\to\widetilde{\Omega}$, where $\Omega$ and $\widetilde{\Omega}$ are domains in $\mathbb R^n$. We prove that mappings with the bounded on the set $\Omega\setminus S$, where a set $S$ has $\sigma$-finite $(n-1)$-measure, geometric $p$-dilatation, are $W^1_{p,\loc}$-- mappings and generate bounded composition operators on Sobolev spaces.
\end{abstract}
\maketitle

\section{Introduction }

Let $\Omega$ and $\widetilde{\Omega}$ be domains in the Euclidean space $\mathbb R^n$, $n\geq 2$. Recall that a homeomorphic mapping $\varphi:\Omega\to\widetilde{\Omega}$ is called quasiconformal, if the conformal capacity inequality
$$
\cp_n\left(\varphi^{-1}(\widetilde{F}_0),\varphi^{-1}(\widetilde{F}_1);\Omega\right)\leq K_n \cp_n\left(\widetilde{F}_0,\widetilde{F}_1;\widetilde{\Omega}\right)
$$
holds for any condenser $(\widetilde{F}_0,\widetilde{F}_1)\subset\widetilde{\Omega}$.
The quasiconformal mappings have the geometric description in the terms of the geometric conformal dilatation \cite{GV60}: the homeomorphic mapping $\varphi:\Omega\to\widetilde{\Omega}$ is quasiconformal, is and only if
 $$
 \limsup_{r\to 0}  H_{\varphi}(x,r)=\limsup_{r\to 0}\frac{L_{\varphi}(x,r)}{l_{\varphi}(x,r)}\leq H<\infty\,\,\text{in}\,\, \Omega,
 $$
where $L_{\varphi}(x,r)=\max\limits_{|x-y|=r}|\varphi(x)-\varphi(y)|$ and $l_{\varphi}(x,y)=\min\limits_{|x-y|=r}|\varphi(x)-\varphi(y)|$.

This result was refined in \cite{Ge60,Ge62}, where it was proved, in particular, that for quasiconformality of $\varphi$ is sufficient
$$
\limsup_{r\to 0}  H_{\varphi}(x,r)\leq H<\infty\,\, \text{in}\,\, \Omega\setminus S,
$$
where a set $S$ has $\sigma$-finite $(n-1)$-measure. Recently the refined geometric characterizations were obtained for mappings which have the integrable geometric conformal dilatations \cite{KM02,KR05}

The homeomorphic mappings $\varphi:\Omega\to\widetilde{\Omega}$ which satisfy the $p$-capacity inequality
\begin{equation}
\label{p_cap}
\cp_p\left(\varphi^{-1}(\widetilde{F}_0),\varphi^{-1}(\widetilde{F}_1);\Omega\right)\leq K_p \cp_p\left(\widetilde{F}_0,\widetilde{F}_1;\widetilde{\Omega}\right),
\,\,1<p<\infty,
\end{equation}
were considered in \cite{Ge69}, where it was proved, that in the case $n-1<p<n$ the mappings $\varphi^{-1}:\widetilde{\Omega}\to\Omega$ are Lipschitz continuous. This result was extended to the case $p=n-1$ in \cite{SSU}.

The homeomorphic mappings which satisfy the capacity inequality (\ref{p_cap}) generate bounded composition operators on Sobolev spaces \cite{GGR95,U93,VU98}.
The bounded composition operators on Sobolev spaces arise in the Sobolev embedding theory \cite{GG94,GS82} and have applications in the weighted Sobolev spaces theory \cite{GU09} and in the spectral theory of elliptic operators \cite{GU17}. In \cite{U93,VU98} were given various characteristics of homeomorphic mappings $\varphi: \Omega\to\widetilde{\Omega}$, where $\Omega, \widetilde{\Omega}$ are domains in $\mathbb R^n$, which generate by the composition rule $\varphi^{\ast}(f)=f\circ\varphi$ the bounded embedding operators on Sobolev spaces:
\begin{equation}
\label{pq}
\varphi^{\ast}:L^1_p(\widetilde{\Omega})\to L^1_q(\Omega),\,\,1<q\leq p<\infty.
\end{equation}

The mappings generate bounded composition operators \eqref{pq} are called as weak $(p,q)$-quasiconformal mappings \cite{GGR95,VU98} because in the case $p=q=n$ we have usual quasiconformal mappings \cite{VG75}. In \cite{U93,VU98} it was proved that the homeomorphic mapping $\varphi:\Omega\to\widetilde{\Omega}$ is the weak $(p,q)$-quasiconformal mapping,  if and only if $\varphi\in W^1_{1,\loc}(\Omega)$, has finite distortion and
$$
K_{p,q}^{\frac{pq}{p-q}}(\varphi;\Omega)=\int\limits_{\Omega}\left(\frac{|D\varphi(x)|^p}{|J(x,\varphi)|}\right)^{\frac{q}{p-q}}~dx<\infty,
\,\,1< q<p< \infty,
$$
and
$$
K_{p,p}^{p}(\varphi;\Omega)=\ess\sup\limits_{\Omega}\frac{|D\varphi(x)|^p}{|J(x,\varphi)|}<\infty,\,\,1< q=p< \infty.
$$
In the case $1< q=p< \infty$ such mappings are called as a weak $p$-quasiconformal mappings \cite{GGR95}.

The first time the geometric $p$-dilatation of weak $p$-quasiconformal mappings  $\varphi:\Omega\to\widetilde{\Omega}$, $p\ne n$, were introduced in \cite{GGR95} (see also \cite{U24}, for detailed proofs):
$$
H_{\varphi, p}^{\lambda}(x,r)=\frac{L_{\varphi}^{p}(x,r)r^{n-p}}{|\varphi(B(x,\lambda r))|},\,\, \lambda\geq 1,
$$
where $|\cdot|$ denoted the $n$-dimensional Lebesgue measure.

The aim of the present work is to give the refined characterizations of weak $p$-quasiconformal mappings in the terms of the geometric $p$-dilatation. We prove that if $\varphi:\Omega\to\widetilde{\Omega}$ is a homeomorphic mapping with
$$
\limsup_{r\to 0} H_{\varphi, p}^{\lambda}(x,r)\leq H_{p}^{\lambda}<\infty,\,\,\text{on}\,\,\Omega\setminus S,
$$
where a set $S$ has $\sigma$-finite $(n-1)$-measure, then $\varphi\in W^1_{p,\loc}(\Omega)$ and generate a bounded composition operator
\begin{equation*}
\varphi^{\ast}:L^1_p(\widetilde{\Omega})\to L^1_p(\Omega),\,\,1< p<\infty.
\end{equation*}

Hence  homeomorphic mappings $\varphi$, with the bounded on the set $\Omega\setminus S$ the geometric $p$-dilatation, satisfy the capacity inequality (\ref{p_cap}) and are Lipschitz mappings in the case $p>n$ \cite{Ge69}.

Remark that quasiconformal mappings can be defined on metric measure spaces, see, for example, \cite{Heinonen,HeiKosk}. The geometric approach
allows to defined weak $p$-quasiconformal mappings on metric measure spaces.

\section{Composition operators on Sobolev spaces}

\subsection{Sobolev spaces}

Let us recall the basic notions of the Sobolev spaces.
Let $\Omega$ be an open subset of $\mathbb R^n$. The Sobolev space $W^1_p(\Omega)$, $1\leq p\leq\infty$, is defined \cite{M}
as a Banach space of locally integrable weakly differentiable functions
$f:\Omega\to\mathbb{R}$ equipped with the following norm:
\[
\|f\mid W^1_p(\Omega)\|=\| f\mid L_p(\Omega)\|+\|\nabla f\mid L_p(\Omega)\|,
\]
where $\nabla f$ is the weak gradient of the function $f$, i.~e. $ \nabla f = (\frac{\partial f}{\partial x_1},...,\frac{\partial f}{\partial x_n})$. The Sobolev space $W^{1}_{p,\loc}(\Omega)$ is defined as a space of functions $f\in W^{1}_{p}(U)$ for every open and bounded set $U\subset  \Omega$ such that $\overline{U}  \subset \Omega$.

The homogeneous seminormed Sobolev space $L^1_p(\Omega)$, $1\leq p\leq\infty$, is defined as a space
of locally integrable weakly differentiable functions $f:\Omega\to\mathbb{R}$ equipped
with the following seminorm:
\[
\|f\mid L^1_p(\Omega)\|=\|\nabla f\mid L_p(\Omega)\|.
\]

In the Sobolev spaces theory, a crucial role is played by capacity as an outer measure associated with Sobolev spaces \cite{M}. In accordance to this approach, elements of Sobolev spaces $W^1_p(\Omega)$ are equivalence classes up to a set of $p$-capacity zero \cite{MH72}.

Recall that a function $f:\Omega\to\mathbb R$ belongs to the class $\ACL(\Omega)$ if it is absolutely continuous on almost all straight lines which are parallel to any coordinate axis. Note that $f$ belongs to the Sobolev space $W^1_{1,\loc}(\Omega)$ if and only if $f$ is locally integrable and it can be changed by a standard procedure (see, e.g. \cite{M} ) on a set of measure
zero (changed by its Lebesgue values at any point where the Lebesgue values exist) so that a modified function belongs to $\ACL(\Omega)$, and its partial derivatives $\frac{\partial f}{\partial x_i}$, $i=1,...,n$, existing a.e., are locally integrable in $\Omega$.

The mapping $\varphi:\Omega\to\mathbb{R}^{n}$ belongs to the Sobolev space $W^1_{p,\loc}(\Omega)$, if its coordinate functions belong to $W^1_{p,\loc}(\Omega)$. In this case, the formal Jacobi matrix $D\varphi(x)$ and its determinant (Jacobian) $J(x,\varphi)$
are well defined at almost all points $x\in\Omega$. The norm $|D\varphi(x)|$ is the operator norm of $D\varphi(x)$. Recall that a mapping $\varphi:\Omega\to\mathbb{R}^{n}$ belongs to $W^1_{p,\loc}(\Omega)$, is a mapping of finite distortion if $D\varphi(x)=0$ for almost all $x$ from $Z=\{x\in\Omega: J(x,\varphi)=0\}$ \cite{VGR}.

\subsection{Composition operators}

Let $\Omega$ and $\widetilde{\Omega}$ be domains in the Euclidean space $\mathbb R^n$. Then a homeomorphic mapping $\varphi:\Omega\to\widetilde{\Omega}$ generates a bounded composition
operator
\[
\varphi^{\ast}:L^1_p(\widetilde{\Omega})\to L^1_q(\Omega),\,\,\,1\leq q\leq p\leq\infty,
\]
by the composition rule $\varphi^{\ast}(f)=f\circ\varphi$, if for
any function $f\in L^1_p(\widetilde{\Omega})$, the composition $\varphi^{\ast}(f)\in L^1_q(\Omega)$
is defined quasi-everywhere in $\Omega$ and there exists a constant $K_{p,q}(\varphi;\Omega)<\infty$ such that
\[
\|\varphi^{\ast}(f)\mid L^1_q(\Omega)\|\leq K_{p,q}(\varphi;\Omega)\|f\mid L^1_p(\widetilde{\Omega})\|.
\]

Recall that the $p$-dilatation \cite{Ge69} of a Sobolev mapping $\varphi: \Omega\to \widetilde{\Omega}$ at the point $x\in\Omega$ is defined as
$$
K_p(x)=\inf \{k(x): |D\varphi(x)|\leq k(x) |J(x,\varphi)|^{\frac{1}{p}}\}.
$$

\begin{thm}
\label{CompTh} Let $\varphi:\Omega\to\widetilde{\Omega}$ be a homeomorphic mapping
between two domains $\Omega$ and $\widetilde{\Omega}$. Then $\varphi$ generates a bounded composition
operator
\[
\varphi^{\ast}:L^1_p(\widetilde{\Omega})\to L^1_{q}(\Omega),\,\,\,1< q\leq p\leq\infty,
\]
 if and only if $\varphi\in W^1_{q,\loc}(\Omega)$
and
\[
K_{p,q}(\varphi;\Omega) := \|K_p \mid L_{\kappa}(\Omega)\|<\infty, \,\,1/q-1/p=1/{\kappa}\,\,(\kappa=\infty, \text{ if } p=q).
\]
The norm of the operator $\varphi^\ast$ is estimated as $\|\varphi^\ast\| \leq K_{p,q}(\varphi;\Omega)$.
\end{thm}

This theorem in the case $p=q=n$ was given in the work \cite{VG75}. The general case $1\leq q\leq p<\infty$ was proved in \cite{U93}, where the weak change of variables formula \cite{H93} was used (see, also the case $n<q=p<\infty$ in \cite{V88}).

\section{Refined geometric characterizations of mappings}

Let a $\varphi:\Omega\to \widetilde{\Omega} $  be a homeomorphic mapping. Recall the notion of the geometric $p$-dilatation, $1<p<\infty$, \cite{GGR95}. Let
$$
H_{\varphi,p}^{\lambda}(x,r)=\frac{L^p_{\varphi}(x,r)r^{n-p}}{|\varphi\left(B(x,\lambda r)\right)|}\,,\,\,\lambda\geq 1,
$$
where $L_{\varphi}(x,r)=\max\limits_{|x-y|=r}|\varphi(x)-\varphi(y)|$. Then {\it the geometric $p$-dilatation} of $\varphi$ at $x$ is defined as
 $$
 H_{\varphi,p}^{\lambda}(x)= \limsup_{r\to 0} \, H_{\varphi,p}^{\lambda}(x,r)\,.
 $$
In the case $\lambda=1$ we will denote the geometric $p$-dilatation by the symbol  $H_{\varphi,p}(x)$.

Recall that a set $S\subset {\mathbb R}^n$ is said to have a {\it $\sigma$-finite $(n-1)$-dimensional measure} \cite{KM02}, if the set $S$ is of the form $S=\cup S_i$ where $H^{n-1}(S_i)<\infty$ and $H^{n-1}$ refers to the $(n-1)$-dimensional Hausdorff measure.

\begin{thm} Let $1<p<\infty$ and   $\varphi:\Omega\to \widetilde{\Omega} $  be a homeomorphic mapping. If
\begin{equation}\label{eqMpHsigma}
\limsup_{r\to 0} \, H_{\varphi,p}(x,r)\leq H_p <\infty  \ {\text for\  each}\  x\in \Omega\setminus S,
\end{equation}
where $S$ has $\sigma$-finite $(n-1)$-measure, then  $\varphi\in {\rm ACL}(\Omega)$.
\end{thm}

\begin{proof} Fix an arbitrary cube $P$, $P\subset\Omega$ with edges parallel to coordinate axes. We
prove that $\varphi$ is absolutely continuous on almost all intersections of $P$ with lines
parallel to the axis $x_n$. Let $P_0$ be the orthogonal projection of $P$ on subspace
$\{x_n = 0\} ={\mathbb R}^{n-1}$  and $I$ be the orthogonal projection of $P$ on the axis $x_n$.
Then $P = P_0 \times I$.

Since $\varphi$ is the homeomorphic mapping then the Lebesgue measure $\Phi(E)=|\varphi(E)|$ induces by the rule $\Phi(A,P)=\Phi(A\times I)$ the monotone countable-additive function defined on measurable subsets of $P_0$. By the Lebesgue theorem on differentiability
(see, for example, \cite{RR55}),  the upper $(n-1)$-dimensional volume derivative
$$
\Phi'(z,P)=\limsup_{r\to 0} \frac{\Phi\left(B^{n-1}(z,r),P\right)}{\omega_{n-1}r^{n-1}}<\infty
$$
for almost all points $z\in  P_0$. Here $B^{n-1}(z,r)$ is an $(n-1)$-dimensional ball with a center at $z\in P_0$ and the radius $r$ and $\omega_{n-1}$ is the $(n-1)$-measure of $(n-1)$-dimensional unit ball.

By the Gross theorem (see e.g. \cite{V71}) for a.e. segments $I$ parallel to some coordinate axis, the set $S\cap I$ is countable.
Let $0<r<\delta$ and $\varepsilon>0$. For each $k=1,2,\ldots$, we define the sets
$$
F_k=\left\{x\in B^{n-1}(z,r)\times \bigcup\limits_{j} I_j:\,
\frac{L^p_{\varphi}(x,r)r^{n-p}}{|\varphi\left(B(x,r)\right)|}\leq \widetilde{H}_p, {\rm for \ all}\  r<\frac{1}{k}\right\}\,,
$$
where a constant $\widetilde{H_p}>H_p$  depends  on $p$ and  $n$ only. The sets $F_k$ are Borel sets,
 $$
 B^{n-1}(z,r)\times \bigcup\limits_{j} I_j\setminus S=\bigcup\limits_{k} F_k,
 $$
for any $k$ there exists an open set $U_k$ such that $F_k\subset U_k $, where $I_j=(a_j,b_j)$, $a_j,b_j\in\mathbb Q$,  and
$$
 |U_k|\leq |F_k| +\frac{\varepsilon}{2^{2k}}\,.
$$
Fix a number $k$. Then for every $x\in F_k$ there exists $r_x>0$ such that

\noindent
(i) $0<r_x< \min\{r, d, |a_j-b_j|\}/10$,

\noindent
(ii) $L^p_{\varphi}(x, r_x) r_x^{n-p}< \widetilde{H}_p |\varphi (B(x,r_x))|$,  and

\noindent
(iii) $B_x\subset U_k$.

By the Besicovitch covering theorem (see, for example, \cite{Fe69}) there exists a countable sequence of  balls $B_1, B_2,\ldots $  from the covering $\{\overline{B}(x,r_x)\}$ so that

$$
B^{n-1}(z,r)\times \bigcup\limits_{j} I_j\subset \bigcup\limits_{j} \overline{B}_j \subset B^{n-1}(z,2r) \times [a-d, b+d],
$$
and $\sum_j \chi_{\overline{B}_j}(x)\leq c(n)$ for every $x\in \mathbb{R}^n$.

For arbitrary number $l\in\mathbb N$ we define the function

$$
\rho(x) = \frac{1}{G}\sum\limits_{i} \frac{L_\varphi(x_i,r_i) }{r_i}\chi_{2B_i}(x)\,,
$$
where $G=\sum\limits_{j=1}^{l} |\varphi(z, b_j)-\varphi(z, a_j)|$. This function $\rho$ is a Borel function, because it is a countable sum of (simple) Borel functions.

Now we estimate the volume integral of the function $\rho$. First of all
\begin{multline*}
\int\limits_{B^{n-1}(z,r)\times \bigcup\limits_{j} I_j} \rho(x)\,  dx \\
\geq\frac{1}{G} \int\limits_{B^{n-1}(z,r)} \,\int\limits_{\bigcup\limits_{j} I_j} \sum\limits_{B_i\cap(\{\zeta\}\times \bigcup\limits_{j} I_j)\neq\emptyset}  \frac{L_{\varphi}(x_i,r_i)}{r_i} \, \chi_{2B_i}(\zeta, x_n)\, dx_n \, d\zeta\,.
\end{multline*}

Note that
$$
\int\limits_{\bigcup\limits_{j} I_j}  \chi_{2B_i}(\zeta, x_n)\, dx_n  \geq \frac{1}{2} \, {\rm diam}(B_i)=r_i
$$
for the balls $B_i$ such that  $B_i\cap(\{\zeta\}\times \bigcup\limits_{j} I_j)\neq\emptyset$. Moreover, for almost every
$\zeta\in B^{n-1}(z,r)$, the sets $\varphi (B_i)$ cover the set
$\varphi\left(\{\zeta\}\times \bigcup\limits_j I_j\right)$ up to a countable set,
because $S$ has $\sigma$-finite $(n-1)$-measure (see Theorem 30.16 in \cite{V71}). Thus, since
$r<\delta$ , we have that
$$
\sum\limits_{B_i\cap(\{\zeta\}\times \bigcup\limits_{j} I_j)\neq\emptyset}  L_{\varphi}(x_i, r_i)  \geq  \frac{1}{2} \sum\limits_{B_i\cap(\{\zeta\}\times \bigcup\limits_{j} I_j)\neq\emptyset} {\rm diam}(\varphi B_i)\geq\frac{1}{8} \, G
$$
for almost every $\zeta\in B^{n-1}(z,r)$.

So
\begin{equation}\label{incrn-1}
\int\limits_{B^{n-1}(z,r)\times \bigcup\limits_{j} I_j} \rho(x)\,  dx \geq c(n) r^{n-1}\,.
\end{equation}

Next we establish the upper bound for the integral in the right side of the inequality (\ref{incrn-1}). Using the monotone
convergence theorem, we obtain the estimate
$$
\int\limits_{B^{n-1}(z,r)\times \bigcup\limits_{j} I_j} \rho(x)\,  dx \leq \frac{c(n)}{G} \sum\limits_{i} L_{\varphi}(x_i, r_i) r_i^{n-1}\,.
$$
Hence, by using the discrete H\"{o}lder inequality, we obtain
\begin{multline*}
\int\limits_{B^{n-1}(z,r)\times \bigcup\limits_{j} I_j} \rho(x)\,  dx \leq \frac{c(n)}{G} \sum\limits_{i} \frac{L_{\varphi}(x_i, r_i)}{|\varphi(B_i)|^\frac{1}{p}}\,  |\varphi(B_i)|^\frac{1}{p} \, r_i^{n-1}
\\
\leq \frac{c(n)}{G} \left(\sum\limits_{i} \left(\frac{L_{\varphi}(x_i, r_i)}{|\varphi(B_i)|^\frac{1}{p}} \,  r_i^{n-1}\right)^{\frac{p}{p-1}} \right)^{\frac{p-1}{p}}\left(\sum\limits_{i} |\varphi(B_i)|\right)^\frac{1}{p}\,.
\end{multline*}
Thus,
\begin{multline*}
\int\limits_{B^{n-1}(z,r)\times \bigcup\limits_{j} I_j} \rho(x)\,  dx \\
\leq \frac{c(n,p)}{G} \left(\sum\limits_{i} \left(\frac{L^p_{\varphi}(x_i, r_i)}{|\varphi(B_i)|}\, r^{n-p}_i\right)^\frac{1}{p-1}  |B_i|\right)^{\frac{p-1}{p}} \left(\sum\limits_{i} |\varphi(B_i)|\right)^\frac{1}{p}\,,
\end{multline*}
where $c(n,p)$ is a positive constant that depends on $n$ and $p$  only. Hence,
$$
\int\limits_{B^{n-1}(z,r)\times \bigcup\limits_{j} I_j} \rho(x)\,  dx \leq \frac{c(n,p)}{G} (\widetilde{H}_p)^\frac{1}{p} \left(\sum\limits_{i} |B_i|\right)^{\frac{p-1}{p}} \left(\sum\limits_{i} |\varphi(B_i)|\right)^\frac{1}{p}\,.
$$

For the last term in this inequality we have that
$$
\sum\limits_{i} |\varphi(B_i)|\leq c(n) |\varphi\left(B^{n-1}(z,2r)\times[a-d, b+d]\right)|=c(n) \Phi\left(B^{n-1}(z,2r)\right)
$$
because the overlapping of the balls was bounded.

Thus,
\begin{equation}\label{introup}
\int\limits_{B^{n-1}(z,r)\times \bigcup\limits_{j} I_j} \rho(x)\,  dx \leq \widetilde{c}(n,p)(\widetilde{H}_p)^\frac{1}{p} G^{-1} \left(\sum\limits_{i}\, |B_i|\right)^\frac{p-1}{p} \left(\Phi\left(B^{n-1}(z,2r)\right)\right)^\frac{1}{p}\,,
\end{equation}
where $\widetilde{c}(n,p)$ is a positive constant that depends on $n$ and $p$  only.

Combining (\ref{incrn-1}) and (\ref{introup}), we have
$$
G\, \leq c(n,p)(\widetilde{H}_p)^\frac{1}{p}  \left(\frac{\sum\limits_{i} |B_i|}{\omega_{n-1}r^{n-1}}\right)^\frac{p-1}{p} \left(\frac{\Phi\left(B^{n-1}(z,2r)\right)}{\omega_{n-1}(2r)^{n-1}}\right)^\frac{1}{p}\,,
$$
where $c(n,p)$ is a some constant.

Since  $ |B_i|\leq \Omega_{n-1} (b_i-a_i) r^{n-1}$, then
\begin{multline*}
\sum\limits_{j=1}^{l} |\varphi(z, b_j)-\varphi(z, a_j)|\, \\
\leq c(n,p)(\widetilde{H}_p)^\frac{1}{p}  \left( \sum\limits_{j=1}^{l} | b_j- a_j|\right)^\frac{p-1}{p} \left(\frac{\Phi\left(B^{n-1}(z,2r)\right)}{\omega_{n-1}(2r)^{n-1}}\right)^\frac{1}{p}\,.
\end{multline*}

Thus, letting $r\to 0$, we get
$$
\sum\limits_{j=1}^{l} |\varphi(z, b_j)-\varphi(z, a_j)|\, \leq c(n,p)(\widetilde{H}_p)^\frac{1}{p}  \left( \sum\limits_{j=1}^{l} | b_j- a_j|\right)^\frac{p-1}{p} \left(\Phi'(z)\right)^\frac{1}{p}\,.
$$
Hence $\varphi\in {\rm ACL}(\Omega)$.
\end{proof}

In the case $\lambda>1$ by using corresponding calculations we have the following assertion.

\begin{thm}
Let $1<p<\infty$ and   $\varphi:\Omega\to \widetilde{\Omega}$  be a homeomorphic mapping. If
\begin{equation}
\limsup_{r\to 0} \, H_{\varphi,p}^{\lambda}(x,r)\leq H_p^{\lambda} <\infty  \ {\text for\  each}\  x\in \Omega\setminus S,
\end{equation}
where $S$ has $\sigma$-finite $(n-1)$-measure, then  $\varphi\in {\rm ACL}(\Omega)$.
\end{thm}

Now we consider differentiability of homeomorphic mappings with a bounded geometric $p$-dilatation.

\begin{thm} Let $1<p<\infty$ and $\varphi:\Omega\to \widetilde{\Omega}$  be a homeomorphic mapping. If
\begin{equation*}
%\label{eqMpHsigma}
\limsup_{r\to 0} \, H_{\varphi,p}(x,r)\leq H_p <\infty  \ {\text for\  each}\  x\in \Omega\setminus S,
\end{equation*}
where $|S|=0$, then $\varphi$ is differentiable almost everywhere in $\Omega$.
\end{thm}

\begin{proof} Let us consider the set function $\Phi (U)=|\varphi(U)|$ defined over the algebra of all the Borel sets $U$
in $\Omega.$ Recall that by the Lebesgue theorem on the differentiability of non-negative, countable-additive finite set functions (see, e.g., \cite{RR55}), there exists a
finite limit for a.e. $x\in \Omega$
\begin{equation}
\varphi'_{v} (x)=\lim\limits_{\varepsilon \rightarrow 0}\frac{\Phi
(B(x,\varepsilon ))}{|B(x,\varepsilon)|},  \label{3.1}
\end{equation}%
where $B(x,\varepsilon )$ is a ball in $\mathbb{R}^{n}$ centered at $x\in \Omega$
with radius $\varepsilon >0.$ The quantity $\varphi'_{v} (x)$ is called the
volume derivative of $\varphi$ at $x.$ \medskip

Now at almost every point $x$ of $\Omega$, by the Lebesgue theorem on the differentiability, $\varphi'_{v} (x)$ exists and $$\limsup_{r\to 0} \, H_{\varphi,p}(x,r)\leq H_p <\infty.$$ Fix such a point $x$. Let $y\in \Omega$ with $0<|x-y|<
d(x,\partial \Omega)$. Then

$$
\frac{|\varphi(y)-\varphi(x)|}{|y-x|}\leq  \left(\omega_n\frac{L_{\varphi}^{p}(x, |x-y|)}{|\varphi\left(B(x,|x-y|)\right)|}\, |x-y|^{n-p}\frac{|\varphi\left(B(x,|x-y|)\right)|}{|B(x,|x-y|)|}\right)^{\frac{1}{p}}\,,
$$
where $\omega_n=|B(0,1)|$.

Letting $y\to x $, we see that
$$
\limsup\limits_{y\rightarrow x} \frac{|\varphi(y)-\varphi(x)|}{|y-x|}\leq \left(\omega_n H_{p} \varphi'_v(x)\right)^\frac{1}{p}<\infty,\,\,\text{for almost all}\,\, x\in\Omega.
$$
Hence by the Rademacher--Stepanov theorem (see, e.g., \cite{Fe69}), the mapping $\varphi$ is differentiable a.e. in $\Omega$ and the theorem follows.
\end{proof}

\begin{thm} Let $1<p<\infty$ and   $\varphi:\Omega\to \widetilde{\Omega} $  be a homeomorphic mapping. If
\begin{equation*}
\limsup_{r\to 0} \, H_{\varphi,p}(x,r)\leq H_p <\infty  \ {\text for\  each}\  x\in \Omega\setminus S,
\end{equation*}
where $S$ has $\sigma$-finite $(n-1)$-measure,
then   $\varphi\in W^{1}_{p,\rm{loc}} (\Omega)$ and
$$
|D\varphi(x)|^p\leq c(n,p)\, H_p |J(x,\varphi)|\  \text{for a.e.} \ x\in \Omega\,,
$$
where $c(n,p)$ is a positive constant that depends on $n$ and $p$  only.

\end{thm}

\begin{proof}
Since $\varphi:\Omega\to {\mathbb R}^n$ is the $\ACL$-mapping differentiable a.e. in $\Omega$, then
$$
\limsup\limits_{r\to 0 } \frac{L_\varphi(x,r)}{r}=\lim\limits_{r\to 0 } \frac{L_\varphi(x,r)}{r}=|D\varphi(x)|\,\,\text{for almost all}\,\, x\in\Omega.
$$
Hence
\begin{multline*}
|D\varphi(x)|^p=\left(\lim\limits_{r\to 0 } \frac{L_\varphi(x,r)}{r} \right)^p \\
\leq  c(n,p) \lim\limits_{r\to 0 } \frac{L_\varphi(x,r)r^{n-p}}{|\varphi(B(x,r))|} \frac{|\varphi(B(x,r))|}{|B(x,r)|}\leq c(n,p)\, H_p \,\varphi'_v(x)
\end{multline*}

So, for any compact set $U\subset \Omega$, we have
$$
\int\limits_{U} |D\varphi(x)|^p \, dx \leq c(n,p)\, H_p\, \int\limits_{U} \varphi'_v(x)\, dx  \leq c(n,p)\, H_p\, |\varphi(U)|<\infty\,.
$$

Therefore $|D\varphi| \in L_{p,\rm{loc}}(\Omega)$ and we have that $\varphi\in W^{1}_{p,\rm{loc}} (\Omega)$.

\end{proof}

Hence, we obtain the following sufficient geometric condition for mappings generate bounded composition operators on Sobolev spaces.

\begin{thm} 
\label{comp}
Let $1<p<\infty$ and   $\varphi:\Omega\to \widetilde{\Omega} $  be a homeomorphic mapping. Suppose
\begin{equation*}
\limsup\limits_{r\to 0}\, H_{\varphi,p}(x,r)\leq H_p <\infty  \ {\text for\  each}\  x\in \Omega\setminus S,
\end{equation*}
where $S$ has $\sigma$-finite $(n-1)$-measure.
Then   $\varphi$ generate by the composition rule $\varphi(f)=f\circ\varphi$ a bounded embedding operator
$$
\varphi^{\ast}: L^1_p(\widetilde{\Omega})\to L^1_p(\Omega).
$$
\end{thm}

Remark, that the necessity of considered geometric conditions for boundedness of composition operators follows from \cite{GGR95}.

Recall the notion of of the variational $p$-capacity associated with Sobolev spaces \cite{GResh}. The condenser in the domain $\Omega\subset \mathbb R^n$ is the pair $(F_0,F_1)$ of connected closed relatively to $\Omega$ sets $F_0,F_1\subset \Omega$. A continuous function $u\in L_p^1(\Omega)$ is called an admissible function for the condenser $(F_0,F_1)$,
if the set $F_i\cap \Omega$ is contained in some connected component of the set $\operatorname{Int}\{x\vert u(x)=i\}$,\ $i=0,1$. We call $p$-capacity of the condenser $(F_0,F_1)$ relatively to domain $\Omega$
the value
$$
{{\cp}}_p(F_0,F_1;\Omega)=\inf\|u\vert L_p^1(\Omega)\|^p,
$$
where the greatest lower bond is taken over all admissible for the condenser $(F_0,F_1)\subset\Omega$ functions. If the condenser have no admissible functions we put the capacity is equal to infinity.

By Theorem~\ref{comp} we obtain \cite{GGR95,U93} the capacity inequality for mappings with the $(n-1)$-almost bounded geometric dilatation.

\begin{thm}
\label{cap}
Let $1<p<\infty$ and   $\varphi:\Omega\to \widetilde{\Omega} $  be a homeomorphic mapping. Suppose
\begin{equation*}
\limsup\limits_{r\to 0}\, H_{\varphi, p}(x,r)\leq H_p <\infty  \ {\text for\  each}\  x\in \Omega\setminus S,
\end{equation*}
where $S$ has $\sigma$-finite $(n-1)$-measure.
Then  the capacity inequality
\begin{equation*}
\cp_p\left(\varphi^{-1}(\widetilde{F}_0),\varphi^{-1}(\widetilde{F}_1);\Omega\right)\leq c(n,p)H_p \cp_p\left(\widetilde{F}_0,\widetilde{F}_1;\widetilde{\Omega}\right),
\,\,1<p<\infty,
\end{equation*}
holds for any condenser $(\widetilde{F}_0,\widetilde{F}_1)\subset \widetilde{\Omega}$.
\end{thm}

Hence, by \cite{Ge69,SSU} we have the following corollary.

\begin{cor}
Let $1<p<\infty$ and   $\varphi:\Omega\to \widetilde{\Omega} $  be a homeomorphic mapping. Suppose
\begin{equation*}
\limsup\limits_{r\to 0}\, H_{\varphi,p}(x,r)\leq H_p <\infty  \ {\text for\  each}\  x\in \Omega\setminus S,
\end{equation*}
where $S$ has $\sigma$-finite $(n-1)$-measure.
Then $\varphi$ is a Lipschitz mapping if $n<p<\infty$, and $\varphi^{-1}$ is a Lipschitz mapping if $n-1\leq p<n$.
\end{cor}

By Theorem \ref{cap} and Theorem 2 in \cite{Ge69} we obtain the next significant result on quasi-isometric mappings.

\begin{thm}
\label{QISOM}
Let $1<p<\infty, \, p\neq n\,,$ and   $\varphi:\Omega\to \widetilde{\Omega} $  be a homeomorphic mapping. Suppose
\begin{equation*}
\limsup\limits_{r\to 0}\, H_{\varphi, p}(x,r)\leq H_p <\infty  \ {\text for\  each}\  x\in \Omega\setminus S,
\end{equation*}
and
\begin{equation*}
\limsup\limits_{r\to 0}\, H_{\varphi^{-1}, p}(y,r)\leq H_p <\infty  \ {\text for\  each}\  y\in \widetilde{\Omega} \setminus \widetilde{S},
\end{equation*}
where $S$ and $\widetilde{S}$ have  $\sigma$-finite $(n-1)$-measure.
Then  $\varphi$ is a quasi-isometric mapping.
\end{thm}

Now, by using the composition duality theorem \cite{U93,VU98} we obtain the following result on mappings with controlled $p$-capacity ($p$-moduli) distortion.

\begin{thm}
\label{cap_dual}
Let $1<p<\infty$ and   $\varphi:\Omega\to \widetilde{\Omega} $  be a homeomorphic mapping. Suppose
\begin{equation*}
\limsup\limits_{r\to 0}\, H_{\varphi,q}(x,r)\leq H_q <\infty  \ {\text for\  each}\  x\in \Omega\setminus S,\,\,q=p\frac{n-1}{p-1},
\end{equation*}
where $S$ has $\sigma$-finite $(n-1)$-measure.
Then  the capacity inequality
\begin{equation*}
\cp_p\left(\varphi(\widetilde{F}_0),\varphi(\widetilde{F}_1);\Omega\right)\leq c(n,p,H_q) \cp_p\left(\widetilde{F}_0,\widetilde{F}_1;\widetilde{\Omega}\right),
\,\,1<p<\infty,
\end{equation*}
holds for any condenser $(\widetilde{F}_0,\widetilde{F}_1)\subset \widetilde{\Omega}$.
\end{thm}

\begin{proof}
Let
\begin{equation*}
\limsup\limits_{r\to 0}\, H_{\varphi,q}(x,r)\leq H_q <\infty  \ \text{ for\  each} \  x\in \Omega\setminus S,\,\,q=p\frac{n-1}{p-1}.
\end{equation*}
Then by Theorem~\ref{cap} the homeomorphic mapping   $\varphi$ generate by the composition rule $\varphi(f)=f\circ\varphi$ a bounded embedding operator
$$
\varphi^{\ast}: L^1_q(\widetilde{\Omega})\to L^1_q(\Omega).
$$
Since $q>n-1$, then by the composition duality theorem \cite{U93,VU98}, the inverse mapping $\varphi^{-1}:\widetilde{\Omega} \to \Omega$ generate a bounded composition operator
$$
\left(\varphi^{-1}\right)^{\ast}: L^1_{q'}(\Omega)\to L^1_{q'}(\widetilde{\Omega}),\,\,q'=\frac{q}{q-(n-1)}.
$$
Because $q'=p$, then by the capacity characterization of composition operator \cite{U93,VU98} we have
\begin{equation*}
\cp_p\left(\varphi(\widetilde{F}_0),\varphi(\widetilde{F}_1);\Omega\right)\leq c(n,p,H_q) \cp_p\left(\widetilde{F}_0,\widetilde{F}_1;\widetilde{\Omega}\right),
\,\,1<p<\infty,
\end{equation*}
for any condenser $(\widetilde{F}_0,\widetilde{F}_1)\subset \widetilde{\Omega}$.
\end{proof}

\vskip 0.2cm

In conclusion we  consider the following geometric property of weak $p$-quasiconformal mappings.

\begin{thm}
Let $\varphi: \Omega\to\widetilde{\Omega}$, $\Omega,\widetilde{\Omega}\subset\mathbb R^n$, be a $p$-quasiconformal mapping, $p>n$. Then
\begin{equation}\label{eqTh1}
\limsup\limits_{r\to 0} \left( \left(\frac{L_\varphi(x,r)}{r}\right)^\frac{p-n}{p-1}-\left(\frac{l_{\varphi}(x,r)}{r}\right)^\frac{p-n}{p-1}\right)\leq c(n,p) K_p^\frac{p}{p-1}
 <\infty \,\,\text{for all}\,\,x\in\Omega\,,
\end{equation}
where $c(n,p)$ is a positive constant which depends only on $n$ and $p$.

\end{thm}

\begin{proof} Because $\varphi: \Omega\to\widetilde{\Omega}$ is a $p$-quasiconformal mapping, $p>n$, then \cite{GGR95}
\begin{equation}\label{eqpr1}
\cp_p(\varphi^{-1}(F_0),\varphi^{-1}(F_1);\Omega)\leq K^p_p \cp_p(F_0,F_1;\widetilde{\Omega}),
\end{equation}
for any condenser $(F_0,F_1)\subset \widetilde{\Omega}$.

Let $x\in \Omega$, $r>0$,  $F_0=\{ y\in \mathbb{R}^n: |y-\varphi(x)|\leq l_\varphi(x,r)\}$
and $F_1=\{ y\in \mathbb{R}^n: |y-\varphi(x)|\geq L_\varphi(x,r)\}$. Then by (2) in \cite{Ge69},
\begin{equation}\label{eqpr3}
\cp_p(F_0,F_1;\Omega)= n\omega_{n}\left( \frac{p-n}{p-1}\right)^{p-1} \left( L^\frac{p-n}{p-1}_{\varphi}(x,r)-l^\frac{p-n}{p-1}_{\varphi}(x,r)\right)^{1-p}\,,
\end{equation}
where $\omega_{n}$ is the volume   of the unit ball  in $\mathbb{R}^n$.

On the other hand, by Lemma 3 in  \cite{Ge69} , we have
\begin{equation}\label{eqpr2}
\cp_p(\varphi^{-1}(F_0),\varphi^{-1}(F_1);\Omega)\geq c(n,p) \, r^{n-p}\,,
\end{equation}
where $c(n,p)$ is a positive constant which depends only on $n$ and $p$.

Combining (\ref{eqpr2}),  (\ref{eqpr3}) with   (\ref{eqpr1}), we obtain
$$
c(n,p) \, r^{n-p} \leq K^p_p\,  \omega_{n-1}\left( \frac{p-n}{p-1}\right)^{p-1} \left(L^\frac{p-n}{p-1}_{\varphi}(x,r)-l^\frac{p-n}{p-1}_{\varphi}(x,r) \right)^{1-p}\,.
$$

Hence,
\begin{equation}\label{proofLl}
\left(\frac{L_\varphi(x,r)}{r}\right)^\frac{p-n}{p-1}-\left(\frac{l_{\varphi}(x,r)}{r}\right)^\frac{p-n}{p-1}\leq c(n,p) K^\frac{p}{p-1}\,,
\end{equation}
where $c(n,p)$ is a positive constant which depends only on $n$ and $p$.

Passing to the upper limit as $r\to 0$ in (\ref{proofLl}), we obtain relation (\ref{eqTh1}).
\end{proof}

\vskip 0.5cm

\noindent
Ruslan Salimov; Institute of Mathematics of NAS of Ukraine, Tereschenkivs'ka Str. 3, 01 601 Kyjiv, Ukraine

\noindent
\emph{E-mail address:} \email{ruslan.salimov1@gmail.com} \\

\noindent
Alexander Ukhlov; Department of Mathematics, Ben-Gurion University of the Negev, P.O.Box 653, Beer Sheva, 8410501, Israel

\noindent
\emph{E-mail address:} \email{ukhlov@math.bgu.ac.il


\begin{thebibliography}{References}



\bibitem{BKS} Z. Balogh, P. Koskela, and S. Rogovin, Absolute continuity of quasiconformal mappings on curves, Geom. Funct. Anal. 17 (2007), N. 3, 645--664.

\bibitem{Fe69} H.~Federer, Geometric measure theory, Sp\-rin\-ger Verlag, Berlin, (1969).

\bibitem{Ge60} F.~W.~Gehring, The definitions and exceptional sets for quasiconformal mappings, Ann. Acad. Sci. Fenn. Ser. AI, 281 (1960), 1--28.

\bibitem{Ge62} F.~W.~Gehring, Rings and quasiconformal mappings in space, Trans. Amer. Math. Soc., 103 (1962), 353--393.

\bibitem{Ge69} F.~W.~Gehring, Lipschitz mappings and the $p$-capacity of rings in $n$-space, Advances in the theory of Riemann surfaces (Proc. Conf., Stony Brook, N. Y., 1969), 175--193. Ann. of Math. Studies, No. 66. Princeton Univ. Press, Princeton, N. J. (1971)

%\bibitem{Ge} F.~W.~Gehring, Lipschitz mappings and the $p$-capacity of ring in $n$-space, Advances in the theory of Riemann surfaces, Ann. Math. Studies, 66 (1971), 175--193.

\bibitem{GV60} F.~W.~Gehring, J.~V\"ais\"al\"a, On the geometric definition for quasiconformal mappings, Commentarii Mathematici Helvetici, 36 (1960), 19--32.

\bibitem{GG94} V.~Gol'dshtein, L.~Gurov, Applications of change of variables operators for exact embedding theorems, Integral Equations Operator Theory 19 (1994), 1--24.

\bibitem{GGR95} V.~Gol'dshtein, L.~Gurov, A.~Romanov, Homeomorphisms that induce monomorphisms of Sobolev spaces,
Israel J. Math., 91 (1995), 31--60.

\bibitem{GResh} V. M. Gol'dshtein, Yu. G. Reshetnyak, Quasiconformal mappings and Sobolev spaces, Dordrecht, Boston, London: Kluwer Academic Publishers, (1990).

\bibitem{GS82} V.~M.~Gol'dshtein, V.~N.~Sitnikov, Continuation of functions of the class $W^1_p$ across H\"older boundaries, Imbedding theorems and their applica-ions, Trudy Sem. S. L. Soboleva, 1 (1982), 31--43.

\bibitem{GU09} V.~Gol'dshtein, A.~Ukhlov, Weighted Sobolev spaces and embedding theorems, Trans. Amer. Math. Soc., 361, (2009), 3829--3850.

\bibitem{GU17} V.~Gol'dshtein, A.~Ukhlov, The spectral estimates for the Neumann-Laplace operator in space domains,
Adv. in Math., 315 (2017), 166--193.

\bibitem{H93} P.~Hajlasz, Change of variable formula under the minimal assumptions, Colloq. Math., 64 (1993), 93--101.

\bibitem{Heinonen} J.~Heinonen, Lectures on analysis on metric spaces, Springer-Verlag, New York, (2001).

\bibitem{HeiKosk} J.~Heinonen, P.~Koskela, Quasiconformal maps in metric spaces with controlled geometry,  Acta Math. 181 (1998), N 1, 1--61.

\bibitem{KM02} S.~Kallunki, O.~Martio, $ACL$ homeomorphisms and linear dilatation, Proc. Amer. Math. Soc., 130(2002), 1073--1078.

\bibitem{KR05} P.~Koskela, S.~Rogovin, Linear dilatation and absolute continuity, Ann. Acad. Sci. Fenn. Math., 30(2005), 385--392.

\bibitem{M} V.~Maz'ya, Sobolev spaces: with applications to elliptic partial differential equations, Springer: Berlin/Heidelberg, (2010).

\bibitem{MH72} V.~Maz'ya, V.~Havin, Nonlinear potential theory, Russian Math. Surveys, 27 (1972), 71--148.

\bibitem{RR55} T.~Rado, P.~V.~Reichelderfer, Continuous Transformations in Analysis. Springer-Verlag, Berlin (1955).

\bibitem{SSU}	R.~Salimov, E.~Sevost'yanov, A.~Ukhlov, Capacity inequalities and Lipschitz continuity of mappings.,Trans. Razmadze Math. Inst., 178 (2024).

\bibitem{U93} A.~D.~Ukhlov, On mappings, which induce embeddings of Sobolev spaces, Siberian Math. J., 34 (1993), 185--192.

\bibitem{U24} A.~D.~Ukhlov, On geometric characterizations of mappings generate composition operators on Sobolev spaces, Ukrainian Math. Bull., 21 (2014).

\bibitem{V71} J.~V\"ais\"al\"a, Lectures on $n$-dimensional quasiconformal mappings. Lecture Notes in Math. 229, Springer Verlag, Berlin, 1971.

\bibitem{V88} S.~K.~Vodop'yanov, Taylor Formula and Function Spaces, Novosibirsk Univ. Press., 1988.

\bibitem{VG75} S.~K.~Vodop'yanov, V.~M.~Gol'dstein, Lattice isomorphisms of the spaces $W^1_n$ and quasiconformal mappings,
Siberian Math. J., 16 (1975), 224--246.

\bibitem{VGR} S.~K.~Vodop'yanov, V.~M.~Gol'dshtein, Yu.~G.~Reshetnyak, On geometric properties of functions with generalized first derivatives, Uspekhi Mat. Nauk {34} (1979), 17--65.

\bibitem{VU98} S.~K.~Vodop'yanov, A.~D.~Ukhlov, Sobolev spaces and $(P,Q)$-quasiconformal mappings of Carnot groups.
Siberian Math. J., 39 (1998), 665--682.



\end{thebibliography}
\end{document}

\bibitem{MS86} V.~Maz'ya, T.~O.~Shaposhnikova, Multipliers in Spaces of Differentiable Functions,
Leningrad Univ. Press., 1986.

\bibitem{RR55} T.~Rado, P.~V.~Reichelderfer, Continuous Transformations in Analysis. Springer-Verlag, Berlin (1955).

\bibitem{Va} Ju.~V\"ais\"al\"a, Lectires on $n$-dimensional quasiconformal mappings, Springer Verlag, (1955).

\bibitem{U94} A.~D.~Ukhlov, Sobolev spaces and mappings of Carnot groups associated with them, PhD Thesis, Novosibirks State University, Novosibirsk,  (1994).

\bibitem{VU04} S.~K.~Vodop'yanov, A.~D.~Ukhlov, Set Functions and Their Applications in the Theory of Lebesgue and Sobolev Spaces. I, Siberian Adv. Math., 14 (2004), 78--125.

\bibitem{Vu} M.~Vuorinen, Conformal Geometry and Quasiregular Mappings, Lecture Notes in Math., vol.1319, Springer-Verlag, Berlin, Heidelberg, New York, 2006.